\documentclass[11pt]{amsart}
\usepackage[margin=1in]{geometry}
\usepackage{amsmath,amssymb,amsthm}
\usepackage{enumitem}
\usepackage{booktabs}
\usepackage{tikz}
\usepackage[colorlinks=true,linkcolor=blue!50!black,citecolor=blue!50!black,urlcolor=blue!50!black]{hyperref}

\definecolor{accA}{HTML}{D9730D}
\definecolor{accB}{HTML}{1F5FBF}
\tikzset{
  body/.style={fill=black!9, draw=black, line width=0.9pt},
  insc/.style={draw=black, dashed, line width=0.55pt},
  diag/.style={draw=black, line width=0.45pt, dash pattern=on 1.2pt off 1.6pt},
  chordline/.style={draw=black, line width=0.5pt, dash pattern=on 2.4pt off 1.8pt},
  thinline/.style={draw=black, line width=0.4pt},
  dimline/.style={draw=black, line width=0.4pt, <->, >=stealth},
  circA/.style={draw=accA, line width=1.05pt},
  circB/.style={draw=accB, line width=1.05pt},
  pt/.style={circle, fill=black, inner sep=0.95pt},
  lab/.style={font=\footnotesize, inner sep=0.8pt},
}

\newtheorem{theorem}{Theorem}[section]
\newtheorem{lemma}[theorem]{Lemma}
\newtheorem{proposition}[theorem]{Proposition}
\newtheorem{corollary}[theorem]{Corollary}
\newtheorem{conjecture}[theorem]{Conjecture}
\theoremstyle{definition}
\newtheorem{definition}[theorem]{Definition}
\newtheorem{remark}[theorem]{Remark}

\newcommand{\cst}{c_*}
\newcommand{\qch}{q_{\mathrm{ch}}}
\newcommand{\vp}{\varphi}
\newcommand{\dist}{\operatorname{dist}}

\title{A note on Kuperberg's quadrilateral conjecture}
\author{Viktor V\'igh} 
\address{Department of Geometry, Bolyai Institute, 
University of Szeged, Aradi v\'ertan\'uk tere 1, 6720 Szeged, Hungary}
\email{vigvik@math.u-szeged.hu}
\date{3 October 2026}

\begin{document}
\maketitle

\begin{abstract}
Kuperberg conjectured in 1983 that every convex body $K$ in the plane is contained in a quadrilateral of area at most
$\frac{3}{\sqrt5}\,|K|$, the extremal bodies being the affine-regular pentagons. On the basis of numerical experiments we
propose a stronger conjecture of a purely polygonal nature: every convex polygon $P$ is contained in a quadrilateral of
area at most $\frac{3}{\sqrt5}\,|P|$ whose sides are parallel to sides or diagonals of $P$. This \emph{chord conjecture}
implies the inequality conjectured by Kuperberg. For pentagons it is a theorem of Hong, Ismailescu, Kwak and Park, of
which we give a short proof by area identities. We also show that the classical bound $\sqrt2$ remains valid for
quadrilaterals with sides parallel to chords; the proof is similar to Ismailescu's proof of this bound, but it starts
from an inscribed quadrilateral of maximal area instead of a circumscribed quadrilateral of minimal area. We close with a discussion of
the difficulties in reaching the constant $\frac{3}{\sqrt5}$.
\end{abstract}

\section{Introduction}\label{sec:intro}

Throughout, a \emph{convex body} is a compact convex set with non-empty interior in the Euclidean plane, $|\cdot|$
denotes area, and all polygons are convex. For $n\ge3$ let $C_n(K)$ be a polygon with at most $n$ sides, of minimal
area, containing the convex body $K$, and put
\[
c_n=\sup_K\frac{|C_n(K)|}{|K|} ,
\]
the supremum being taken over all convex bodies. Thus $c_n$ is the smallest constant such that every convex body $K$
is contained in a polygon with at most $n$ sides of area at most $c_n|K|$. Note that the ratio $|C_n(K)|/|K|$ is invariant
under affine transformations.

The corresponding problem for inscribed polygons is completely solved: by a theorem of Sas~\cite{Sas}, every convex
body $K$ contains an $n$-gon of area at least $\frac{n}{2\pi}\sin\frac{2\pi}{n}\,|K|$, with equality only for
ellipses. About circumscribed polygons much less is known; for the following facts we refer to the introduction of
the recent paper of Fodor and Grundbacher~\cite{FG}. For triangles, $c_3=2$, and the extremal bodies are the
parallelograms (Gross~\cite{Gross}, Eggleston~\cite{Egg}, Chakerian~\cite{Ch}). For $n\ge12$ the best known bounds
are $\frac n\pi\tan\frac\pi n\le c_n\le\frac{n-2}\pi\tan\frac\pi{n-2}$, due to L.~Fejes T\'oth~\cite{FT}; the lower
bound is the value of the ratio for ellipses. For $5\le n\le11$ the best known upper bound is
$c_n\le\sec\frac\pi n$, due to Ismailescu~\cite{Ism}.

This note is about quadrilaterals, and we write $q(K)=|C_4(K)|$. Chakerian and Lange~\cite{CL} proved that
$c_4\le\sqrt2$ (see also~\cite{Ch}); Kuperberg~\cite{Kup} showed that $c_4<\sqrt2$, without an explicit bound, and
made the following conjecture.

\begin{conjecture}[Kuperberg~\cite{Kup}]\label{conj:K}
Every convex body $K$ satisfies $q(K)\le\frac{3}{\sqrt5}\,|K|$, and equality holds only for affine-regular pentagons.
\end{conjecture}

Here $\cst:=\frac3{\sqrt5}=1.3416\ldots$ is the value of $q(P)/|P|$ for the regular pentagon $P$; one of its minimal
quadrilaterals is shown in Figure~\ref{fig:intro}(a). So, in contrast to the inscribed problem, the extremal body is
expected to be a polygon: for an ellipse $q(K)=\frac4\pi|K|\approx1.273\,|K|$, and for centrally symmetric bodies
$q(K)\le\frac43|K|$ by a theorem of Petty~\cite{Petty} on circumscribed parallelograms. The conjecture is known for
pentagons: Hong, Ismailescu, Kwak and Park~\cite{HIKP} proved that every pentagon $P$ is contained in a quadrilateral
of area at most $\frac3{\sqrt5}|P|$, with equality only for affine-regular pentagons; in particular
$c_4\ge\frac3{\sqrt5}$. For general convex bodies no explicit bound below $\sqrt2$ was known until recently, when
Fodor and Grundbacher~\cite{FG} proved
\[
q(K)<(1-2.6\cdot10^{-7})\,\sqrt2\,|K|
\]
by means of stability versions of the steps of Ismailescu's proof~\cite{Ism} of the bound $\sqrt2$.

\begin{figure}[t]
\centering
\begin{tikzpicture}[scale=1.550, line join=round, line cap=round, baseline=(current bounding box.center)]
  \fill[accB!14] (0.000,1.000)--(-1.176,1.000)--(-0.951,0.309)--cycle;
  \fill[accB!14] (0.000,1.000)--(1.176,1.000)--(0.951,0.309)--cycle;
  \draw[body] (0.000,1.000)--(-0.951,0.309)--(-0.588,-0.809)--(0.588,-0.809)--(0.951,0.309)--cycle;
  \draw[chordline] (-0.951,0.309)--(0.951,0.309);
  \draw[circB] (-1.176,1.000)--(-0.588,-0.809)--(0.588,-0.809)--(1.176,1.000)--cycle;
  \node[pt] at (0.000,1.000) {};
  \node[pt] at (-0.951,0.309) {};
  \node[pt] at (-0.588,-0.809) {};
  \node[pt] at (0.588,-0.809) {};
  \node[pt] at (0.951,0.309) {};
  \node[lab] at (0.000,1.160) {$A_1$};
  \node[lab] at (-1.161,0.309) {$A_2$};
  \node[lab] at (-0.758,-0.929) {$A_3$};
  \node[lab] at (0.758,-0.929) {$A_4$};
  \node[lab] at (1.161,0.309) {$A_5$};
  \node[lab] at (0.000,-0.220) {$P$};
  \node[lab, accB!75!black] at (1.000,-0.420) {$T$};
  \node[lab] at (0.000,-1.220) {(a)};
\end{tikzpicture}
\hspace{1.5cm}
\begin{tikzpicture}[scale=1.330, line join=round, line cap=round, baseline=(current bounding box.center)]
  \fill[accB!14] (-0.990,0.840)--(-1.333,0.813)--(-0.840,-0.350)--cycle;
  \fill[accB!14] (-0.840,-0.350)--(-0.364,-1.473)--(-0.100,-1.260)--cycle;
  \fill[accB!14] (0.940,-0.420)--(1.423,-0.030)--(0.980,0.270)--cycle;
  \draw[body] (0.020,0.920)--(-0.990,0.840)--(-0.840,-0.350)--(-0.100,-1.260)--(0.940,-0.420)--(0.980,0.270)--cycle;
  \draw[chordline] (-0.990,0.840)--(-0.100,-1.260);
  \draw[circB] (0.020,0.920)--(-1.333,0.813)--(-0.364,-1.473)--(1.423,-0.030)--cycle;
  \node[pt] at (0.020,0.920) {};
  \node[pt] at (-0.990,0.840) {};
  \node[pt] at (-0.840,-0.350) {};
  \node[pt] at (-0.100,-1.260) {};
  \node[pt] at (0.940,-0.420) {};
  \node[pt] at (0.980,0.270) {};
  \node[lab] at (0.070,1.080) {$B_1$};
  \node[lab] at (-0.960,1.030) {$B_2$};
  \node[lab] at (-1.110,-0.350) {$B_3$};
  \node[lab] at (0.060,-1.430) {$B_4$};
  \node[lab] at (1.070,-0.600) {$B_5$};
  \node[lab] at (1.200,0.370) {$B_6$};
  \node[lab] at (0.200,-0.050) {$P$};
  \node[lab, accB!75!black] at (-0.980,-0.800) {$Q$};
  \node[lab] at (0.000,-1.780) {(b)};
\end{tikzpicture}
\caption{(a) The regular pentagon $P$ and one of its minimal quadrilaterals $T$: three sides of $T$ contain sides of
$P$, and the fourth passes through $A_1$ and is parallel to the diagonal $A_2A_5$ (dashed); $|T|=\frac3{\sqrt5}|P|$.
(b) A hexagon $P$ and its chord quadrilateral $Q$ of least area, $|Q|=1.1947\ldots|P|$: three sides of $Q$ contain
the sides $B_6B_1$, $B_1B_2$, $B_4B_5$, and the fourth passes through $B_3$ and is parallel to the diagonal $B_2B_4$.}
\label{fig:intro}
\end{figure}
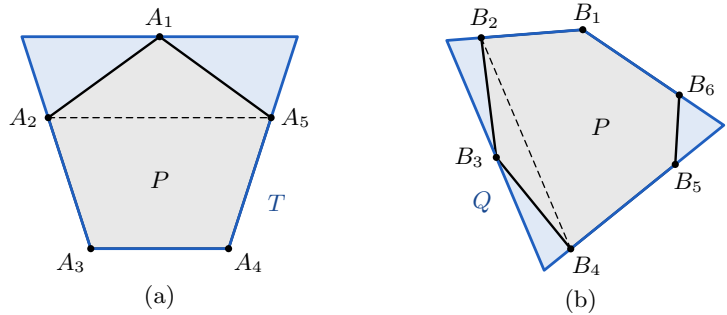

In this note we propose a strengthening of Conjecture~\ref{conj:K} in which only polygons occur and the sides of the
quadrilateral are restricted to finitely many directions, those of the sides and diagonals of the polygon
(Section~\ref{sec:chord}). We show that it implies the inequality of Conjecture~\ref{conj:K}, observe that for
pentagons it is what was proved in~\cite{HIKP}, and give a short proof of that theorem (Section~\ref{sec:pentagon}).
In Section~\ref{sec:sqrt2} we prove that the bound $\sqrt2$ holds in the restricted setting as well. The last section
summarises some difficulties that a proof of the chord conjecture has to overcome.

\section{The chord conjecture}\label{sec:chord}

\begin{definition}\label{def:chord}
Let $P$ be a polygon. A \emph{chord} of $P$ is a segment joining two vertices of $P$; so the chords are the sides and
the diagonals. A \emph{chord quadrilateral} of $P$ is a polygon with at most four sides that contains $P$ and each of
whose sides is parallel to a chord of $P$. We denote by $\qch(P)$ the least area of a chord quadrilateral of $P$.
\end{definition}

An $n$-gon has $\binom n2$ chords. Translating a side of a chord quadrilateral towards $P$ until it touches $P$ does not
increase the area, so the minimum is attained by one of finitely many quadrilaterals, whose sides lie on support lines
of $P$ parallel to chords. Each of these candidates is determined by the vertices of $P$, and its area is a rational
function of their coordinates. Clearly $q(P)\le\qch(P)$, the ratio $\qch(P)/|P|$ is invariant under affine
transformations, and $\qch(P)=|P|$ if $P$ has at most four vertices. Figure~\ref{fig:intro} shows two examples.
The inequality $q(P)\le\qch(P)$ can be strict; an example is given in Remark~\ref{rem:compare}.

\begin{conjecture}[chord conjecture]\label{conj:chord}
Every polygon $P$ satisfies $\qch(P)\le\frac3{\sqrt5}\,|P|$, and equality holds only for affine-regular pentagons.
\end{conjecture}

The following proposition is almost obvious, but we include it for the sake of completeness.

\begin{proposition}\label{prop:implies}
Let $c\ge1$. If $\qch(P)\le c\,|P|$ for every polygon $P$, then $q(K)\le c\,|K|$ for every convex body $K$. In
particular, Conjecture~\ref{conj:chord} implies the inequality of Conjecture~\ref{conj:K}.
\end{proposition}

\begin{proof}
Let $K$ be a convex body, and let $P_1\subset P_2\subset\cdots$ be polygons with vertices on the boundary of $K$ whose
union is dense in $K$. Since $q(P_k)\le\qch(P_k)\le c\,|P_k|$, there are polygons $Q_k\supset P_k$ with at most four
sides and $|Q_k|\le c\,|P_k|\le c\,|K|$. Let $B\subset P_1$ be a disc with centre $o$ and radius $r>0$. If $x\in Q_k$,
then $Q_k$ contains the convex hull of $B$ and $x$, whose area is at least $r\,|x-o|$; hence $|x-o|\le c\,|K|/r$, and
the polygons $Q_k$ lie in a fixed disc. Passing to a subsequence we may assume that the (at most four) vertices of
$Q_k$ converge. Then $Q_k$ converges in the Hausdorff metric to the convex hull $Q$ of the limit points, which is a
polygon with at most four sides. For every $j$ we have $P_j\subset Q_k$ for all $k\ge j$, hence $P_j\subset Q$;
therefore $K\subset Q$. Finally $|Q|=\lim|Q_k|\le c\,|K|$ by the continuity of the area.
\end{proof}

The proof shows the well-known fact that it suffices to prove the inequality of Conjecture~\ref{conj:K} for polygons,
and it uses the chord condition only through $q\le\qch$. The characterisation of the equality case in
Conjecture~\ref{conj:K} does not follow from Conjecture~\ref{conj:chord} in this way.

Although Conjecture~\ref{conj:chord} is stronger than the inequality of Conjecture~\ref{conj:K}
(Proposition~\ref{prop:implies}), it may be more accessible. The sides of the candidates have prescribed
directions, so for an $n$-gon the conjecture is a finite system of inequalities between rational functions of the
coordinates of the vertices, of the kind that has been handled for $n=5$ (Section~\ref{sec:pentagon}). For $q(P)$
there are also finitely many candidates, by DePano's theorem (see~\cite[Lemma~2.1]{HIKP}), but one side of a candidate
may be determined by a midpoint condition instead of a direction. Chord quadrilaterals
are inherited from sub-polygons: if $P'$ is the convex hull of some of the vertices of $P$, then every chord
quadrilateral of $P'$ that contains $P$ is a chord quadrilateral of $P$. The extremal quadrilaterals of the regular
pentagon are chord quadrilaterals, and so are the quadrilaterals used in the proofs for pentagons.

\paragraph{Numerical experiments.}
For a given polygon, $\qch(P)$ is computed by enumerating all candidates. We did this for $16\,830$ random
polygons with $5\le n\le10$ vertices, drawn from three random models (Table~\ref{tab:numerics}). Random polygons are
far from extremal, so we also maximised $\qch(P)/|P|$ numerically over the $n$-gons for $n=6,7,8$, by a trust-region
bundle method for minima of smooth functions. The $50$ runs were started from random points on an ellipse, from
perturbed affine-regular $n$-gons, and from perturbed affine-regular pentagons with $n-5$ additional vertices; a lower
bound on the areas of the triangles spanned by three consecutive vertices kept the polygons genuine $n$-gons. Every
run ended either at an affine-regular $n$-gon ($n=6,7$), which appears to be a local maximum, with
$\qch(P)/|P|=\frac43$ and $1.2588\ldots$, respectively, or at a nearly degenerate $n$-gon, in which some triangle
spanned by three consecutive vertices has area at most $10^{-3}|P|$: close to an affine-regular pentagon, with a value
slightly below $\cst$ (the largest was $\cst-2.3\cdot10^{-7}$), or, in three runs, close to an affine-regular hexagon,
with a value slightly below $\frac43$. Since $q\le\qch$ and $q$ is continuous, the supremum of $\qch(P)/|P|$ over the
$n$-gons is at least $\cst$ for every $n\ge5$; the experiments suggest that it equals $\cst$ and is not attained for
$n\ge6$.

\begin{table}[t]
\centering\small
\begin{tabular}{@{}cccccccc@{}}
\toprule
 & \multicolumn{2}{c}{random polygons} & & \multicolumn{4}{c}{ascent runs, by the polygon at which they stopped}\\
\cmidrule(lr){2-3}\cmidrule(l){5-8}
$n$ & number & $\max\qch(P)/|P|$ & & number & affine-regular $n$-gon & near a pentagon & near a hexagon\\
\midrule
$5$ & $6000$ & $1.3247$ & & & & & \\
$6$ & $6000$ & $1.3183$ & & $30$ & $11$ & $19$ & \\
$7$ & $3000$ & $1.3163$ & & $12$ & $3$ & $7$ & $2$\\
$8$ & $1200$ & $1.2678$ & & $8$ & $0$ & $7$ & $1$\\
$9$ & $450$ & $1.2517$ & & & & & \\
$10$ & $180$ & $1.2619$ & & & & & \\
\bottomrule
\end{tabular}
\caption{Numerical experiments supporting Conjecture~\ref{conj:chord}. Left: random polygons. Right: numerical maximisation of
$\qch(P)/|P|$ over $n$-gons. A run ends at an affine-regular $n$-gon, or at a nearly degenerate $n$-gon close to an
affine-regular pentagon (values slightly below $\cst=1.34164\ldots$) or to an affine-regular hexagon (values
slightly below $\frac43$). No polygon with $\qch(P)>\cst|P|$ was found.}
\label{tab:numerics}
\end{table}

\section{Pentagons}\label{sec:pentagon}

Let $P=A_1A_2A_3A_4A_5$ be a pentagon with vertices in counterclockwise order; indices are taken modulo $5$. The
\emph{ear} at $A_i$ is the triangle $A_{i-1}A_iA_{i+1}$, and we put
\[
a_i=|A_{i-1}A_iA_{i+1}|,\qquad S=|P|,\qquad t_3=|A_2A_3A_5|,\qquad t_4=|A_2A_4A_5| .
\]
Let $\ell$ be the line through $A_1$ parallel to the diagonal $A_2A_5$, and let $T=T(P;A_1)$ be the quadrilateral
bounded by $\ell$ and the lines $A_2A_3$, $A_3A_4$, $A_4A_5$, with vertices $A_3$, $A_4$, $Y$ and $X$ (Figure~\ref{fig:pentagon}(a)). Recall that
$\cst=\frac3{\sqrt5}$, and put
\[
\kappa=\cst-1,\qquad\vp=\frac{1+\sqrt5}2 .
\]

\begin{theorem}[Hong, Ismailescu, Kwak and Park~\cite{HIKP}]\label{thm:pentagon}
Let $P=A_1A_2A_3A_4A_5$ be a pentagon labelled so that its ear at $A_1$ is a smallest one: $a_1\le a_i$ for all $i$.
Then $P\subset T$ and
\[
|T|\le\cst\,|P| ,
\]
with equality if and only if $P$ is affine-regular.
\end{theorem}

In~\cite{HIKP} this is Theorem~3.1 together with Observation~3.2; there the pentagon is $ABCDE$, labelled so that
$DEA$ is a smallest ear, and $T$ is the quadrilateral $BCFG$. Three sides of $T$ contain sides of $P$ and the fourth
is parallel to a diagonal, so $T$ is a chord quadrilateral of $P$. Since $q(P)=\cst|P|$ for the regular
pentagon~\cite{Kup,HIKP}, we obtain:

\begin{corollary}\label{cor:pentagon}
The chord conjecture holds for pentagons: $\qch(P)\le\cst\,|P|$ for every pentagon $P$, with equality if and only if
$P$ is affine-regular.
\end{corollary}

In~\cite{HIKP} the theorem is proved in affine coordinates: the inequality becomes a polynomial inequality in four
variables, which is certified by a decomposition into non-negative terms found with the help of computer algebra. We
give a simplified proof using two area identities, which might be easier to check by hand. First we prove a lemma.

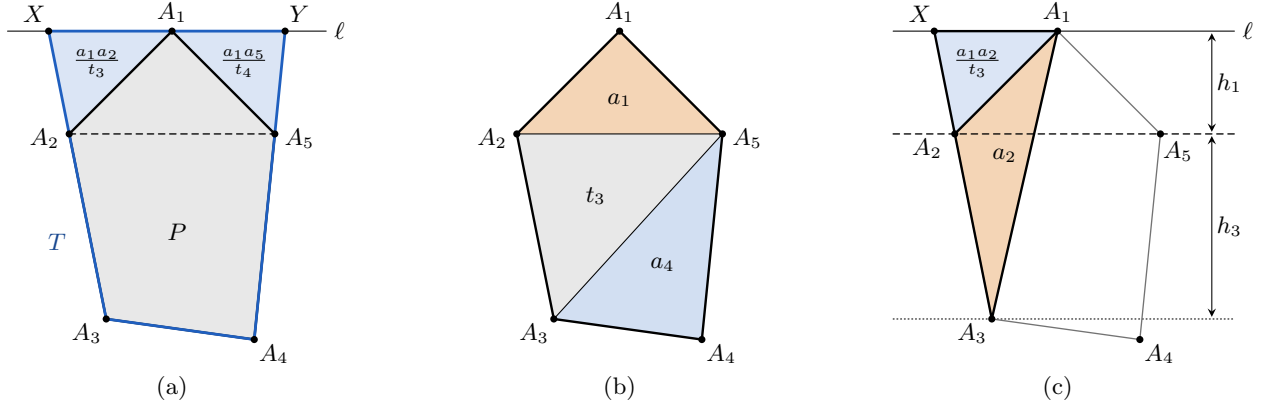
\begin{figure}[t]
\centering
\begin{tikzpicture}[scale=1.360, line join=round, line cap=round, baseline=(current bounding box.center)]
  \fill[accB!16] (0.000,1.000)--(-1.200,1.000)--(-1.000,0.000)--cycle;
  \fill[accB!16] (0.000,1.000)--(1.100,1.000)--(1.000,0.000)--cycle;
  \draw[thinline] (-1.600,1.000)--(1.500,1.000);
  \node[lab] at (1.630,1.000) {$\ell$};
  \draw[body] (0.000,1.000)--(-1.000,0.000)--(-0.640,-1.800)--(0.800,-2.000)--(1.000,0.000)--cycle;
  \draw[chordline] (-1.000,0.000)--(1.000,0.000);
  \draw[circB] (-1.200,1.000)--(-0.640,-1.800)--(0.800,-2.000)--(1.100,1.000)--cycle;
  \node[pt] at (0.000,1.000) {};
  \node[pt] at (-1.000,0.000) {};
  \node[pt] at (-0.640,-1.800) {};
  \node[pt] at (0.800,-2.000) {};
  \node[pt] at (1.000,0.000) {};
  \node[lab] at (0.000,1.180) {$A_1$};
  \node[lab] at (-0.830,-1.940) {$A_3$};
  \node[lab] at (0.990,-2.140) {$A_4$};
  \node[lab] at (-1.240,-0.040) {$A_2$};
  \node[lab] at (1.240,-0.040) {$A_5$};
  \node[pt] at (-1.200,1.000) {};
  \node[pt] at (1.100,1.000) {};
  \node[lab] at (-1.340,1.170) {$X$};
  \node[lab] at (1.240,1.170) {$Y$};
  \node[lab] at (0.050,-0.950) {$P$};
  \node[lab, accB!75!black] at (-1.120,-1.050) {$T$};
  \node[lab, font=\scriptsize] at (-0.730,0.690) {$\tfrac{a_1a_2}{t_3}$};
  \node[lab, font=\scriptsize] at (0.700,0.690) {$\tfrac{a_1a_5}{t_4}$};
  \node[lab] at (0.000,-2.470) {(a)};
\end{tikzpicture}
\hfill
\begin{tikzpicture}[scale=1.360, line join=round, line cap=round, baseline=(current bounding box.center)]
  \fill[accA!30] (1.000,0.000)--(0.000,1.000)--(-1.000,0.000)--cycle;
  \fill[black!9] (-1.000,0.000)--(-0.640,-1.800)--(1.000,0.000)--cycle;
  \fill[accB!20] (-0.640,-1.800)--(0.800,-2.000)--(1.000,0.000)--cycle;
  \draw[thinline] (-1.000,0.000)--(1.000,0.000);
  \draw[thinline] (-0.640,-1.800)--(1.000,0.000);
  \draw[body, fill=none] (0.000,1.000)--(-1.000,0.000)--(-0.640,-1.800)--(0.800,-2.000)--(1.000,0.000)--cycle;
  \node[pt] at (0.000,1.000) {};
  \node[pt] at (-1.000,0.000) {};
  \node[pt] at (-0.640,-1.800) {};
  \node[pt] at (0.800,-2.000) {};
  \node[pt] at (1.000,0.000) {};
  \node[lab] at (0.000,1.180) {$A_1$};
  \node[lab] at (-0.830,-1.940) {$A_3$};
  \node[lab] at (0.990,-2.140) {$A_4$};
  \node[lab] at (-1.240,-0.040) {$A_2$};
  \node[lab] at (1.240,-0.040) {$A_5$};
  \node[lab] at (0.000,0.333) {$a_1$};
  \node[lab] at (-0.233,-0.600) {$t_3$};
  \node[lab] at (0.417,-1.267) {$a_4$};
  \node[lab] at (0.000,-2.470) {(b)};
\end{tikzpicture}
\hfill
\begin{tikzpicture}[scale=1.360, line join=round, line cap=round, baseline=(current bounding box.center)]
  \fill[accB!16] (0.000,1.000)--(-1.200,1.000)--(-1.000,0.000)--cycle;
  \fill[accA!30] (0.000,1.000)--(-1.000,0.000)--(-0.640,-1.800)--cycle;
  \draw[thinline] (-1.600,1.000)--(1.720,1.000);
  \node[lab] at (1.850,1.000) {$\ell$};
  \draw[chordline] (-1.600,0.000)--(1.720,0.000);
  \draw[thinline, densely dotted] (-1.600,-1.800)--(1.720,-1.800);
  \draw[draw=black!55, line width=0.5pt] (0.000,1.000)--(-1.000,0.000)--(-0.640,-1.800)--(0.800,-2.000)--(1.000,0.000)--cycle;
  \draw[body, fill=none] (0.000,1.000)--(-1.200,1.000)--(-0.640,-1.800)--cycle;
  \draw[body, fill=none] (0.000,1.000)--(-1.000,0.000);
  \node[pt] at (0.000,1.000) {};
  \node[pt] at (-1.000,0.000) {};
  \node[pt] at (-0.640,-1.800) {};
  \node[pt] at (0.800,-2.000) {};
  \node[pt] at (1.000,0.000) {};
  \node[lab] at (0.000,1.180) {$A_1$};
  \node[lab] at (-0.830,-1.940) {$A_3$};
  \node[lab] at (0.990,-2.140) {$A_4$};
  \node[lab] at (-1.270,-0.140) {$A_2$};
  \node[lab] at (1.170,-0.170) {$A_5$};
  \node[pt] at (-1.200,1.000) {};
  \node[lab] at (-1.340,1.170) {$X$};
  \draw[dimline] (1.500,0.020)--(1.500,0.980);
  \node[lab] at (1.670,0.500) {$h_1$};
  \draw[dimline] (1.500,-0.020)--(1.500,-1.780);
  \node[lab] at (1.670,-0.900) {$h_3$};
  \node[lab, font=\scriptsize] at (-0.780,0.690) {$\tfrac{a_1a_2}{t_3}$};
  \node[lab] at (-0.517,-0.207) {$a_2$};
  \node[lab] at (0.000,-2.470) {(c)};
\end{tikzpicture}
\caption{The proof of Theorem~\ref{thm:pentagon}. (a) A pentagon whose smallest ear is at $A_1$, and the
quadrilateral $T=XA_3A_4Y$; the difference $T\setminus P$ consists of two triangles, of areas $a_1a_2/t_3$ and
$a_1a_5/t_4$. (b) The dissection $S=a_1+t_3+a_4$. (c) Lemma~\ref{lem:id}(a): the triangles
$A_1XA_2$ and $A_1A_2A_3$ have a common apex and collinear bases, and $|XA_2|:|A_2A_3|=h_1:h_3=a_1:t_3$, where $h_1$
and $h_3$ are the distances of $A_1$ and of $A_3$ from the line $A_2A_5$.}
\label{fig:pentagon}
\end{figure}


\begin{lemma}[area identities]\label{lem:id}
With the notation above, the following holds.
\begin{enumerate}[label=\textup{(\alph*)},leftmargin=2em,itemsep=1pt,topsep=3pt]
\item $|A_1XA_2|=a_1a_2/t_3$ and $|A_1YA_5|=a_1a_5/t_4$; hence $|T|=S+\dfrac{a_1a_2}{t_3}+\dfrac{a_1a_5}{t_4}$.
\item $a_1S=(a_1+t_3)(a_1+t_4)-a_2t_4-a_5t_3$.
\end{enumerate}
\end{lemma}

\begin{proof}
(a) The triangles $A_1XA_2$ and $A_1A_2A_3$ have the common apex $A_1$, and their bases $XA_2$ and $A_2A_3$ lie on one
line, so $|A_1XA_2|:a_2=|XA_2|:|A_2A_3|$. Since $X$, $A_2$, $A_3$ are collinear and $A_2$ lies on the line $A_2A_5$,
\[
\frac{|XA_2|}{|A_2A_3|}=\frac{\dist(X,A_2A_5)}{\dist(A_3,A_2A_5)}=\frac{\dist(A_1,A_2A_5)}{\dist(A_3,A_2A_5)}
=\frac{|A_2A_5A_1|}{|A_2A_5A_3|}=\frac{a_1}{t_3} ,
\]
where the second equality holds because $XA_1$ is parallel to $A_2A_5$ (Figure~\ref{fig:pentagon}(c)). The second
formula is the mirror image of the first, and the formula for $|T|$ follows from \eqref{eq:T}.

(b) Both sides are multiplied by the square of the determinant under an affine map, so we may assume $A_1=(0,1)$,
$A_2=(-1,0)$, $A_5=(1,0)$. Then $a_1=1$, $A_3=(x_3,-t_3)$ and $A_4=(x_4,-t_4)$, and a direct computation gives
\[
a_2=\tfrac12(1+x_3+t_3),\qquad a_5=\tfrac12(1-x_4+t_4),\qquad S=1+\tfrac12(t_3+t_4+x_4t_3-x_3t_4) ,
\]
whence $(1+t_3)(1+t_4)-a_2t_4-a_5t_3=S$.
\end{proof}

\begin{proof}[Proof of Theorem~\ref{thm:pentagon}]

We use the notation of Figure~\ref{fig:pentagon}, $T=XA_3A_4Y$ and obviously \begin{equation}\label{eq:T}
    |T|=S+|A_1XA_2|+|A_1YA_5|.
\end{equation}

All quantities are areas, so we may normalize $a_1=1$ by a dilation. By
Lemma~\ref{lem:id}(a) and~(b),
\[
|T|-S=\frac{a_2}{t_3}+\frac{a_5}{t_4}=\frac{a_2t_4+a_5t_3}{t_3t_4}=\frac{(1+t_3)(1+t_4)-S}{t_3t_4} ,
\]
and since $\cst S-|T|=\kappa S-(|T|-S)$,
\begin{equation}\label{eq:reform}
t_3t_4\bigl(\cst S-|T|\bigr)=S\,(1+\kappa\,t_3t_4)-(1+t_3)(1+t_4) .
\end{equation}
The hypothesis on $a_1$ being minimal yields some constraints on $S,t_3,t_4$. Figure~\ref{fig:pentagon}(b) shows that \begin{equation}\label{eq:triv}
    a_3=S-1-t_4 \quad \text{and} \quad a_4=S-1-t_3,
\end{equation}  so
the conditions $a_3\ge1$, $a_4\ge1$ say that
\begin{equation}\label{eq:far}
S\ge2+\max(t_3,t_4) .
\end{equation}
By Lemma~\ref{lem:id}(b), the conditions $a_2\ge1$, $a_5\ge1$ at the two neighbours of $A_1$ give
\begin{equation}\label{eq:near}
S=(1+t_3)(1+t_4)-a_2t_4-a_5t_3\le(1+t_3)(1+t_4)-t_4-t_3=1+t_3t_4 .
\end{equation}
In view of \eqref{eq:reform} $\cst S \ge |T|$ is equivalent to $S\,(1+\kappa\,t_3t_4)\ge (1+t_3)(1+t_4)$, so it remains to prove the following: \emph{if $S,t,t'>0$ satisfy
$2+\max(t,t')\le S\le1+tt'$, then}
\begin{equation}\label{eq:three}
S\,(1+\kappa\,tt')\ge(1+t)(1+t') ,
\end{equation}
\emph{with equality if and only if $t=t'=\vp$ and $S=2+\vp$.}

We may assume $t\le t'$. First, $2+t'\le S\le1+tt'\le1+t'^2$ gives $t'^2-t'-1\ge0$, that is, $t'\ge\vp$. Next, the
left-hand side of \eqref{eq:three} increases with $S$ and $S\ge2+t'$, hence
\begin{equation}\label{eq:D}
S\,(1+\kappa\,tt')-(1+t)(1+t')\ \ge\ (2+t')(1+\kappa\,tt')-(1+t)(1+t').
\end{equation}

With $D(t'):=\kappa t' (t'+2)-(t'+1)$ we obtain $(2+t')(1+\kappa\,tt')-(1+t)(1+t')=1+tD(t').$ If $D(t')\ge 0$ we are done. If $D(t')<0$ then $0\le t\le t'$ implies  \[1+tD(t')\ge 1+t'D(t')=\kappa t'^3+(2\kappa-1)t'^2-t'+1=\kappa(t'-\vp)\left( t'^2+\frac{5-\sqrt 5}{4}t'-\frac{5+\sqrt 5}{4}\right).\]
It is easy to check that for $t'\ge \vp$ the quadratic expression is positive, and obviously $t'-\vp\ge 0$, hence we proved \eqref{eq:three}. Equality can hold only if $t'=\vp$, $t=t'$ and $S=2+t'$. Conversely, for these values
$S=1+\vp^2$, the constraints hold and equality holds in \eqref{eq:three}.


It remains to discuss equality in the theorem. If $|T|=\cst|P|$, then $t_3=t_4=\vp$ and $S=2+\vp$; so $a_3=a_4=1$ by \eqref{eq:triv} , and $\vp\,(a_2+a_5)=(1+\vp)^2-(2+\vp)=2\vp$ by Lemma~\ref{lem:id}~(b), whence $a_2=a_5=1$ because
$a_2,a_5\ge1$. Thus all five ears are equal. A pentagon with equal ears is affine-regular: $a_i=a_{i+1}$ means that
$A_{i-1}A_{i+2}$ is parallel to $A_iA_{i+1}$; in the normalisation $A_1=(0,1)$, $A_2=(-1,0)$, $A_5=(1,0)$ the
parallelisms $A_5A_3\parallel A_1A_2$, $A_2A_4\parallel A_1A_5$ and $A_3A_4\parallel A_2A_5$ give
$A_3=(1-\lambda,-\lambda)$ and $A_4=(\lambda-1,-\lambda)$ with some $\lambda>0$, and $A_1A_4\parallel A_2A_3$ gives
$\lambda^2-\lambda-1=0$, i.e.\ $\lambda=\vp$; so the pentagon is unique up to affine maps, and the regular pentagon
has equal ears. Conversely, for an affine-regular pentagon, normalised so that all ears are $1$,
\eqref{eq:triv} and Lemma~\ref{lem:id}~(b) give $t_3=t_4=S-2$ and $S^2-5S+5=0$, hence $S=2+\vp$ (as $S>2$) and $t_3=t_4=\vp$, and
equality holds in \eqref{eq:three} and in the theorem.
\end{proof}

{\it Remark.} The two halves of the hypothesis enter separately, through \eqref{eq:far} and \eqref{eq:near}, and neither can be
dropped. In the coordinates used for Lemma~\ref{lem:id}(b), the pentagon with $A_3=(-\frac23,-\frac32)$,
$A_4=(\frac23,-\frac32)$ has $a_1\le a_3,a_4$ and $|T|=\frac{85}{63}|P|>\cst|P|$, and the pentagons with
$A_3=(-\varepsilon,-1-\varepsilon)$, $A_4=(\varepsilon,-1-\varepsilon)$ have $a_1\le a_2,a_5$ and $|T|/|P|\to2$ as
$\varepsilon\to0$.

{\it Remark.} Substituting \eqref{eq:triv} into Lemma~\ref{lem:id}(b) one obtains
$S^2-S\sum_ia_i+\sum_ia_ia_{i+1}=0$, which is Gauss's pentagon formula (see~\cite{Pech}); so (b) is just that formula in
variables adapted to the vertex $A_1$.

\section{The bound \texorpdfstring{$\sqrt2$}{sqrt 2} for chord quadrilaterals}\label{sec:sqrt2}

The following result is the chord variant of the classical inequality by Chakerian and Lange~\cite{CL}.

\begin{theorem}\label{thm:sqrt2}
Let $P$ be a polygon with at least four vertices, and let $m=m(P)$ be the largest area of a polygon with at most four
vertices contained in $P$. Then
\[
\qch(P)\ \le\ \min\Bigl\{\,2m,\ \frac{|P|^2}{m}\Bigr\}\ \le\ \sqrt2\,|P| .
\]
\end{theorem}

We will use the mixed area $V(K,L)$ of two convex bodies, defined by $|K+L|=|K|+2V(K,L)+|L|$; see~\cite{Sch}. It is
symmetric, monotone in each argument with respect to inclusion, $V(K,K)=|K|$, and it satisfies Minkowski's inequality
$V(K,L)^2\ge|K|\,|L|$. If $X$ is a polygon with side lengths $L_e$ and outer unit normals $n_e$, then
$V(X,K)=\frac12\sum_eL_e\,h_K(n_e)$, where $h_K(u)=\max_{x\in K}\langle x,u\rangle$ is the support function of $K$. The following is essentially Theorem 1 in \cite{Ch}. 

\begin{lemma}[Chakerian \cite{Ch}]\label{lem:mink}
Let $K$ be a convex body, let $X\subset K$ be a polygon with side lengths $L_e$ and outer unit normals $n_e$, and let
$Y$ be the intersection of the supporting half-planes $\{x:\langle x,n_e\rangle\le h_K(n_e)\}$ of $K$; so $Y$ is the
polygon circumscribed about $K$ whose sides are parallel to those of $X$. Then
\[
|X|\,|Y|\ \le\ V(X,K)^2\ \le\ |K|^2 ,\qquad\text{where}\quad V(X,K)=|X|+\tfrac12\sum_eL_e\,\delta_e
\]
and $\delta_e=h_K(n_e)-h_X(n_e)\ge0$ is the distance between the line of the side $e$ of $X$ and the parallel support
line of $K$ beyond it.
\end{lemma}

\begin{proof}
Every line $\{x:\langle x,n_e\rangle=h_K(n_e)\}$ bounds $Y$ and contains a point of $K\subset Y$; hence
$h_Y(n_e)=h_K(n_e)$ for all $e$, and therefore $V(X,Y)=\frac12\sum_eL_eh_K(n_e)=V(X,K)$. Minkowski's inequality gives
$|X|\,|Y|\le V(X,Y)^2=V(X,K)^2$, and $V(X,K)\le V(K,K)=|K|$ because $X\subset K$. The formula for $V(X,K)$ follows
from $|X|=\frac12\sum_eL_eh_X(n_e)$.
\end{proof}

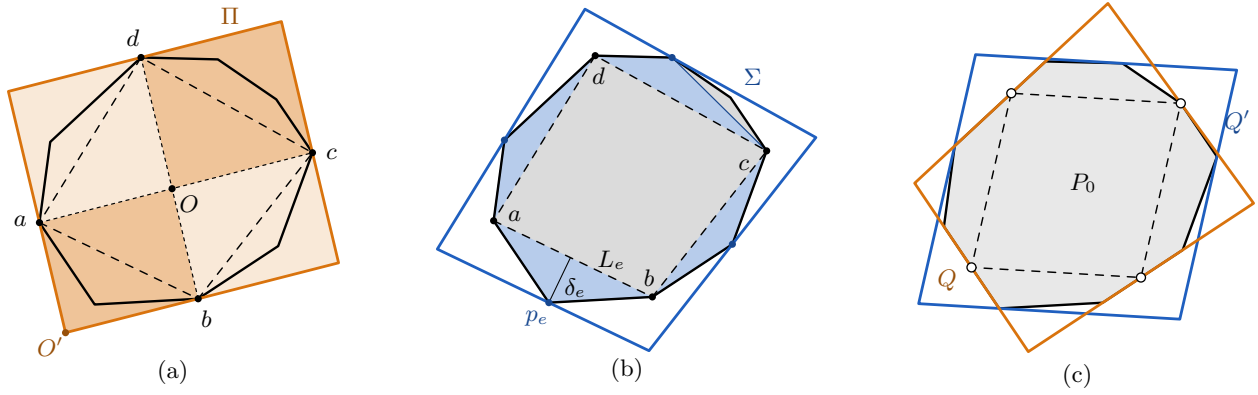
\begin{figure}[t]
\centering
\begin{tikzpicture}[scale=0.840, line join=round, line cap=round, baseline=(current bounding box.center)]
  \fill[accA!42] (-2.000,-0.700)--(-1.589,-2.434)--(0.500,-1.900)--cycle;
  \fill[accA!16] (0.500,-1.900)--(2.711,-1.334)--(2.300,0.400)--cycle;
  \fill[accA!42] (2.300,0.400)--(1.811,2.466)--(-0.400,1.900)--cycle;
  \fill[accA!16] (-0.400,1.900)--(-2.489,1.366)--(-2.000,-0.700)--cycle;
  \fill[accA!42] (-2.000,-0.700)--(0.089,-0.166)--(0.500,-1.900)--cycle;
  \fill[accA!16] (0.500,-1.900)--(0.089,-0.166)--(2.300,0.400)--cycle;
  \fill[accA!42] (2.300,0.400)--(0.089,-0.166)--(-0.400,1.900)--cycle;
  \fill[accA!16] (-0.400,1.900)--(0.089,-0.166)--(-2.000,-0.700)--cycle;
  \draw[body, fill=none] (-2.000,-0.700)--(-1.130,-1.990)--(0.500,-1.900)--(1.760,-1.070)--(2.300,0.400)--(1.730,1.240)--(0.810,1.870)--(-0.400,1.900)--(-1.830,0.570)--cycle;
  \draw[insc] (-2.000,-0.700)--(0.500,-1.900)--(2.300,0.400)--(-0.400,1.900)--cycle;
  \draw[diag] (-2.000,-0.700)--(2.300,0.400);
  \draw[diag] (0.500,-1.900)--(-0.400,1.900);
  \draw[circA] (-1.589,-2.434)--(2.711,-1.334)--(1.811,2.466)--(-2.489,1.366)--cycle;
  \node[pt] at (-2.000,-0.700) {};
  \node[lab] at (-2.300,-0.680) {$a$};
  \node[pt] at (0.500,-1.900) {};
  \node[lab] at (0.630,-2.210) {$b$};
  \node[pt] at (2.300,0.400) {};
  \node[lab] at (2.600,0.380) {$c$};
  \node[pt] at (-0.400,1.900) {};
  \node[lab] at (-0.530,2.210) {$d$};
  \node[pt] at (0.089,-0.166) {};
  \node[lab] at (0.359,-0.406) {$O$};
  \node[pt, fill=accA!70!black] at (-1.589,-2.434) {};
  \node[lab, accA!70!black] at (-1.829,-2.654) {$O'$};
  \node[lab, accA!70!black] at (1.000,2.550) {$\Pi$};
  \node[lab] at (0.100,-3.050) {(a)};
\end{tikzpicture}
\hfill
\begin{tikzpicture}[scale=0.840, line join=round, line cap=round, baseline=(current bounding box.center)]
  \fill[black!14] (-2.000,-0.700)--(-1.130,-1.990)--(0.500,-1.900)--(1.760,-1.070)--(2.300,0.400)--(1.730,1.240)--(0.810,1.870)--(-0.400,1.900)--(-1.830,0.570)--cycle;
  \fill[accB!32] (-2.000,-0.700)--(-1.130,-1.990)--(0.500,-1.900)--cycle;
  \fill[accB!32] (0.500,-1.900)--(1.760,-1.070)--(2.300,0.400)--cycle;
  \fill[accB!32] (2.300,0.400)--(0.810,1.870)--(-0.400,1.900)--cycle;
  \fill[accB!32] (-0.400,1.900)--(-1.830,0.570)--(-2.000,-0.700)--cycle;
  \draw[draw=accB!75!black, line width=0.45pt] (-2.000,-0.700)--(-1.130,-1.990)--(0.500,-1.900)--(1.760,-1.070)--(2.300,0.400)--(0.810,1.870)--(-0.400,1.900)--(-1.830,0.570)--cycle;
  \draw[body, fill=none] (-2.000,-0.700)--(-1.130,-1.990)--(0.500,-1.900)--(1.760,-1.070)--(2.300,0.400)--(1.730,1.240)--(0.810,1.870)--(-0.400,1.900)--(-1.830,0.570)--cycle;
  \draw[insc] (-2.000,-0.700)--(0.500,-1.900)--(2.300,0.400)--(-0.400,1.900)--cycle;
  \draw[circB] (-2.887,-1.147)--(0.447,-2.747)--(3.076,0.611)--(-0.561,2.632)--cycle;
  \draw[thinline] (-0.790,-1.281)--(-1.130,-1.990);
  \node[lab] at (-0.716,-1.752) {$\delta_e$};
  \node[lab] at (-0.146,-1.313) {$L_e$};
  \node[pt, fill=accB!75!black] at (-1.130,-1.990) {};
  \node[pt, fill=accB!75!black] at (1.760,-1.070) {};
  \node[pt, fill=accB!75!black] at (0.810,1.870) {};
  \node[pt, fill=accB!75!black] at (-1.830,0.570) {};
  \node[lab, accB!75!black] at (-1.314,-2.234) {$p_e$};
  \node[pt] at (-2.000,-0.700) {};
  \node[lab] at (-1.670,-0.630) {$a$};
  \node[pt] at (0.500,-1.900) {};
  \node[lab] at (0.450,-1.590) {$b$};
  \node[pt] at (2.300,0.400) {};
  \node[lab] at (1.940,0.180) {$c$};
  \node[pt] at (-0.400,1.900) {};
  \node[lab] at (-0.350,1.590) {$d$};
  \node[lab, accB!75!black] at (2.080,1.560) {$\Sigma$};
  \node[lab] at (0.100,-3.050) {(b)};
\end{tikzpicture}
\hfill
\begin{tikzpicture}[scale=0.840, line join=round, line cap=round, baseline=(current bounding box.center)]
  \draw[body] (-2.000,-0.700)--(-1.130,-1.990)--(0.500,-1.900)--(1.760,-1.070)--(2.300,0.400)--(1.730,1.240)--(0.810,1.870)--(-0.400,1.900)--(-1.830,0.570)--cycle;
  \draw[insc] (-0.939,1.399)--(-1.566,-1.344)--(1.103,-1.503)--(1.730,1.240)--cycle;
  \draw[circB] (-2.398,-1.915)--(1.715,-2.159)--(2.611,1.763)--(-1.502,2.007)--cycle;
  \draw[circA] (0.583,2.814)--(-2.461,-0.017)--(-0.671,-2.671)--(2.877,-0.334)--cycle;
  \node[pt, fill=white, draw=black, line width=0.5pt, inner sep=1.2pt] at (-0.939,1.399) {};
  \node[pt, fill=white, draw=black, line width=0.5pt, inner sep=1.2pt] at (-1.566,-1.344) {};
  \node[pt, fill=white, draw=black, line width=0.5pt, inner sep=1.2pt] at (1.103,-1.503) {};
  \node[pt, fill=white, draw=black, line width=0.5pt, inner sep=1.2pt] at (1.730,1.240) {};
  \node[lab] at (0.193,-0.014) {$P_0$};
  \node[lab, accA!70!black] at (-1.950,-1.550) {$Q$};
  \node[lab, accB!75!black] at (2.620,0.950) {$Q'$};
  \node[lab] at (0.100,-3.050) {(c)};
\end{tikzpicture}
\caption{The proof of Theorem~\ref{thm:sqrt2} for a nonagon $P$ (thick outline). (a) A largest inscribed quadrilateral $abcd$ and the
parallelogram $\Pi$ through its vertices with sides parallel to its diagonals. Each corner triangle of $\Pi$ is the
image of one of the triangles $aOb$, $bOc$, $cOd$, $dOa$ under a half-turn (equal shading), so $|\Pi|=2\,|abcd|$.
(b) The quadrilateral $\Sigma$ circumscribed about $P$ with sides parallel to those of $abcd$; each of its sides
contains a vertex of $P$ (dots, such as $p_e$). The area of $abcd$ together with the four shaded triangles is
$V(abcd,P)\le|P|$, and $|abcd|\,|\Sigma|\le V(abcd,P)^2$. (c) For comparison, the configuration of the proof
in~\cite{FG}: a quadrilateral $Q$ of minimal area containing $P$, the parallelogram $P_0$ spanned by the midpoints of
its sides (white dots; they belong to $P$), and the circumscribed parallelogram $Q'$ with sides parallel to those of
$P_0$. In all panels the inscribed quadrilateral is dashed, the circumscribed quadrilateral of twice its area is
orange, and the circumscribed quadrilateral with sides parallel to it is blue. Here $|\Pi|=1.406\,|P|$,
$|\Sigma|=1.373\,|P|$, $|Q|=1.204\,|P|$ and $|Q'|=1.326\,|P|$.}
\label{fig:sqrt2}
\end{figure}

Geometrically, $V(X,K)$ is the area of $X$ together with the triangles erected over its sides whose apexes $p_e$ are
points of $K$ on the corresponding sides of $Y$ (Figure~\ref{fig:sqrt2}(b)); the monotonicity of the mixed area makes
it unnecessary to check that these triangles do not overlap. Inscribed polygons of maximal area and circumscribed
polygons with parallel sides are classical tools for problems of this kind: applied to a largest inscribed $n$-gon,
the lemma gives, together with the theorem of Sas, the bound $c_n\le\frac{2\pi}n\csc\frac{2\pi}n$ of
Chakerian~\cite{Ch}.

\begin{proof}[Proof of Theorem~\ref{thm:sqrt2}]
For points $x,y,z,w$ of $P$ let $f(x,y,z,w)=\frac12\det(z-x,\,w-y)$. This is the area of $xyzw$ if $xyzw$ is a
quadrilateral with vertices in counterclockwise order. As $f$ is an affine function of each of its arguments, the maximum is
attained at a quadruple of vertices of $P$; we fix such a quadruple $a,b,c,d$.
 These points are distinct, otherwise their
convex hull would be a triangle of area at least $m$ spanned by vertices of $P$, and adding a fourth vertex of $P$ would
give a quadrilateral of larger area. Hence $abcd$ is a quadrilateral of area $m$ with vertices in counterclockwise order. Its four
sides and two diagonals are chords of $P$.


First we construct a circumscribed chord parallelogram. For fixed $b,c,d$ the function $x\mapsto f(x,b,c,d)$ is affine, its level lines are
parallel to the diagonal $bd$, and on $P$ it is maximal at $x=a$. Hence the line through $a$ parallel to $bd$ is a
support line of $P$. The same holds at $b$, $c$ and $d$, and the four support lines bound a parallelogram $\Pi\supset
P$ whose sides are parallel to the diagonals $ac$ and $bd$. Let $O$ be the intersection point of the diagonals and
$O'$ the vertex of $\Pi$ beyond the side $ab$. Then $aO'bO$ is a parallelogram, so the triangle $aO'b$ is the image of
the triangle $aOb$ under the half-turn about the midpoint of $ab$; the same holds for the other three sides, and
therefore $|\Pi|=2m$ (Figure~\ref{fig:sqrt2}(a)).

Now, we construct another circumscribed chord quadrilateral. Let $\Sigma$ be the polygon $Y$ of Lemma~\ref{lem:mink} for $K=P$ and $X=abcd$: the
intersection of the four supporting half-planes of $P$ whose boundary lines are parallel to the sides of $abcd$. It is
bounded, because the outer normals of a quadrilateral positively span the plane, and it has at most four sides. By
Lemma~\ref{lem:mink}, $m\,|\Sigma|\le|P|^2$.

Finally, both $\Pi$ and $\Sigma$ are chord quadrilaterals of $P$. Hence $\qch(P)\le|\Pi|=2m$ and
$\qch(P)\le|\Sigma|\le|P|^2/m$, and
\[
\qch(P)^2\ \le\ |\Pi|\,|\Sigma|\ =\ 2m\,|\Sigma|\ \le\ 2\,|P|^2 . \qedhere
\]
\end{proof}

\begin{remark}[comparison with earlier proofs]\label{rem:compare}
As a bound for $q$, Theorem~\ref{thm:sqrt2} reproduces the classical inequality $c_4\le\sqrt2$ of Chakerian and
Lange~\cite{CL,Ch}. We also do not claim that the argument is completely new; our point is that it produces chord quadrilaterals.
It is close to the proof presented in~\cite[Proposition~2.3]{FG}, which is due to Ismailescu~\cite{Ism}, see also \cite{Ch}, but the
starting point is different.
One can check that for the pentagon $P$ with vertices $(0,5)$, $(-5,1)$, $(-3,-4)$, $(3,-4)$, $(5,2)$, for instance, the earlier proof does not give a chord quadrilateral. Also note that here $q(P)=\frac{582}{451}|P|<\qch(P)=\frac{1854}{1435}|P|$.

\end{remark}

\begin{remark}[sharpness of the method]\label{rem:sharp}
The two quadrilaterals $\Pi$ and $\Sigma$ cannot give a constant below $\sqrt2$. For the regular octagon $P$ the
quadrilateral $abcd$ is the square spanned by alternate vertices, and $|\Pi|=|\Sigma|=\sqrt2\,|P|$, although
$\qch(P)=\frac{1+\sqrt2}2|P|\approx1.207\,|P|$.
\end{remark}

\section{Concluding remarks}\label{sec:difficulties}




To attack Kuperberg's conjecture, by Proposition~\ref{prop:implies}, it suffices to consider polygons, but with arbitrarily many vertices. For fixed $n$, a compactness argument modulo affine maps shows that the supremum of $q(P)/|P|$, or of $\qch(P)/|P|$, over the $n$-gons
is attained at a local maximum or is approached by $n$-gons degenerating to a polygon with fewer vertices (for $\qch$
one uses that every chord of the limit is a limit of chords). So the inequalities of both conjectures would follow if
the ratio did not exceed $\cst$ at any local maximum with $n\ge6$ vertices. In our numerical experiments with $\qch$ and $n\le8$
the only local maxima found were the affine-regular hexagons and heptagons, but we cannot exclude others for all $n$. An induction
on $n$ by deleting a vertex fails: if $P_v$ denotes the convex hull of the vertices of $P$ other than $v$, the
inequality $\qch(P)\le\qch(P_v)+\cst\,(|P|-|P_v|)$, which would make the induction work, is false for the regular
hexagon $P$, where $\qch(P)=\frac43|P|$ and the right-hand side is
$\bigl(1+\frac1{2\sqrt5}\bigr)|P|\approx1.224\,|P|$; the same holds with $q$ in place of $\qch$.


The restriction to finitely many directions makes the
candidates explicit, which is an attraction of the chord conjecture, but it removes the variational tools. Sides cannot be rotated, so
there is no midpoint property (see~\cite[Proposition 2.1]{FG}), and there is no analogue of DePano's theorem (see~\cite[Lemma~2.1]{HIKP}): in $4740$
random polygons a chord quadrilateral of least area always had three sides containing sides of $P$, as in
Figure~\ref{fig:intro}, but for the hexagon with vertices $(-14,2)$, $(-4,-10)$, $(11,-14)$, $(12,2)$, $(3,14)$,
$(-8,7)$ the chord quadrilateral of least area is unique, and only two of its sides contain sides of $P$. Moreover
$\qch(P)/|P|$ is not a continuous function of $P$: a new short side brings a new direction. Compactness and perturbation arguments therefore do not carry over
directly, and one has to work with the combinatorics of the vertices.

\section{Declaration}

This work was developed through continuous collaboration between the author and AI models. OpenAI and Anthropic models contributed to proof development, exact computation, literature research, manuscript preparation and review. The author takes full responsibility for all mathematical results and the contents of this manuscript.

\section{Acknowledgement}

This research was supported by NKFIH project no. 150151. Project no. 150151 has been implemented with the support provided by the Ministry of Culture and Innovation of Hungary from the National Research, Development and Innovation Fund, financed under the ADVANCED\_24 funding scheme.

\end{document}